\documentclass[11pt]{amsart}

\usepackage{epigamath}

\usepackage[english]{babel}

\numberwithin{equation}{section}

\newtheorem{thm}{Theorem}[section]

\newtheorem{lem}[thm]{Lemma}

\theoremstyle{definition}

\newcommand{\I}{\mathbf 1}
\newcommand{\C}{\mathbb C}
\newcommand{\sO}{\mathcal O}

\newcommand{\nd}{\nobreakdash-\hspace{0pt}}

\DeclareMathOperator{\End}{End}

\DeclareMathOperator{\Mod}{Mod}

\DeclareMathOperator{\Rep}{Rep}

\EpigaVolumeYear{5}{2021} \EpigaArticleNr{16} \ReceivedOn{March 17,
2021}
\InFinalFormOn{September 14, 2021}
\AcceptedOn{October 27, 2021}

\title{\'Etale triviality of finite
equivariant vector bundles}
\titlemark{\'Etale triviality of finite equivariant bundles}

\author{Indranil Biswas}
\address{School of Mathematics, Tata Institute of Fundamental
Research, Homi Bhabha Road, Mumbai 400005, India}
\email{indranil@math.tifr.res.in}

\author{Peter O'Sullivan}
\address{Mathematical Sciences Institute, The Australian National University,
Canberra ACT 2601, Australia}
\email{peter.osullivan@anu.edu.au}

\authormark{I. Biswas and P. O'Sullivan}

\AbstractInEnglish{Let $H$ be a complex Lie group acting holomorphically on a complex analytic space $X$
such that the restriction to $X_{\mathrm{red}}$ of every $H$\nd invariant regular function on $X$ is constant.
We prove that an $H$\nd equivariant holomorphic vector bundle $E$ over $X$ is $H$\nd finite, meaning
$f_1(E)\,=\, f_2(E)$ as $H$\nd equivariant bundles
for two distinct polynomials $f_1$ and $f_2$ whose coefficients are nonnegative integers,
if and only if the pullback of $E$ along some $H$\nd equivariant finite \'etale covering of $X$
is trivial as an $H$\nd equivariant bundle.}



\TitleInFrench{\'Etale trivialit\'e des fibr\'es vectoriels \'equivariants finis}

\AbstractInFrench{Soit $H$ un groupe de Lie complexe agissant holomorphiquement sur un espace analytique complexe $X$ de telle sorte que la restriction \`a $X_{\mathrm{red}}$ de toute fonction r\'eguli\`ere $H$\nd invariante soit constante. Nous montrons qu'un fibr\'e vectoriel holomorphe $H$\nd invariant $E$ sur $X$ est $H$\nd fini, c'est-\`a-dire v\'erifie $f_1(E)\,=\, f_2(E)$ comme $H$\nd fibr\'es \'equivariants pour deux polyn\^omes distincts $f_1$ et $f_2$ \`a coefficients entiers positifs, si et seulement si l'image r\'eciproque de $E$ par un rev\^etement \'etale fini $H$\nd \'equivariant est triviale comme $H$\nd fibr\'e \'equivariant.}

\begin{document}

%%%%%%%%%%%%%%%%%%%%%%%%%%%%%%%
% Title page
%%%%%%%%%%%%%%%%%%%%%%%%%%%%%%%

%\removeabove{}
%\removebetween{}
%\removebelow{}

\maketitle

\begin{prelims}

\DisplayAbstractInEnglish

\bigskip

\DisplayKeyWords

\medskip

\DisplayMSCclass

\bigskip

\languagesection{Fran\c{c}ais}

\bigskip

\DisplayTitleInFrench

\medskip

\DisplayAbstractInFrench

\end{prelims}

%%%%%%%%%%%%%%%%%%%%%
% Table of Contents
%%%%%%%%%%%%%%%%%%%%%

\newpage

\setcounter{tocdepth}{1}

\tableofcontents

%%%%%%%%%%%%%%%%%%%%%
% Content begins here
%%%%%%%%%%%%%%%%%%%%%

\section{Introduction}

Given a vector bundle $E$, and a polynomial $f(t)\,=\, \sum_{i=0}^n a_it^i$, where $a_i$ are nonnegative integers,
define the vector bundle
$f(E)\,:=\, \bigoplus_{i=0}^n (E^{\otimes i})^{\oplus a_i}$. A vector bundle $E$ is called finite if there are
two such distinct polynomials $f_1$ and $f_2$ satisfying the
condition that the vector bundles $f_1(E)$ and $f_2(E)$ are isomorphic. This notion was
introduced by Weil in \cite{We}. He proved that a vector bundle $E\, \longrightarrow\, X$
over a complex projective variety is finite if there is a
finite \'etale covering $\varpi\, :\, \widehat{X}\, \longrightarrow\, X$ such that
the pulled back vector bundle $\varpi^*E$ is trivial. Nori in his
thesis proved a converse of it. He proved that a vector bundle $E$ over a projective variety $X$,
defined over an algebraically closed field of characteristic zero, is finite if and only if
there is a finite \'etale covering $\varpi\, :\, \widehat{X}\, \longrightarrow\, X$ such that $\varpi^*E$ is trivial
\cite{No1}, \cite{No2}. Nori's proof involves reducing the problem to the case of curves by taking hyperplanes
in $X$. Note that there is a finite \'etale covering $\varpi\, :\, \widehat{X}\, \longrightarrow\, X$
such that $\varpi^*E$ is trivial if and only if $E$ admits a reduction of structure group to a finite group.

Attempts were made to prove the above theorem of Nori in the context of complex manifolds. 
A vector bundle $E$ over a compact K\"ahler manifold $X$ is finite if and only if
there is a finite \'etale covering $\varpi\, :\, \widehat{X}\, \longrightarrow\, X$ such that
$\varpi^*E$ is trivial \cite{BHS}. When $X$ is a compact complex manifold
admitting a Gauduchon astheno-K\"ahler metric, the same result was proved in \cite{BP}. We recall that
astheno-K\"ahler metrics, introduced in \cite{JY}, are weaker than K\"ahler metrics. More recently, in
\cite{Bi} this result was proved for holomorphic vector bundles over any compact complex manifold.

Our aim here is to prove a similar result in a more general setting:
\begin{itemize}
\item We allow $X$ to be a complex analytic space, so it can have singularities and may not be reduced,

\item $X$ is allowed to be noncompact, and

\item a complex Lie group $H$ acts on $X$, and the vector bundles considered are $H$-equivariant.
\end{itemize}

Let $H$ be a complex Lie group acting holomorphically on a complex analytic space $X$. An
$H$\nd equivariant holomorphic vector bundle ${\mathcal V}$ over $X$ is called
$H$\nd finite if there are two distinct polynomials $f_1$ and $f_2$, whose coefficients are nonnegative
integers, such that $f_1(E)$ and $f_2(E)$ are isomorphic as $H$\nd equivariant bundles,
and $H$\nd trivial if it is $H$\nd equivariantly isomorphic to ${\mathcal O}^{\oplus r}_X$ for some $r$.

We prove the following theorem (see Theorem \ref{t:main}):

\begin{thm}\label{ti}
Let $H$ be a complex Lie group acting holomorphically on a complex analytic space $X$
such that the restriction to $X_{\mathrm{red}}$ of every $H$\nd invariant regular function on $X$ is constant.
Then an $H$\nd equivariant holomorphic vector bundle over $X$ is $H$\nd finite
if and only if its pullback along some $H$\nd equivariant finite \'etale covering of $X$
is $H$\nd trivial.
\end{thm}

As mentioned above, when $X$ is a compact complex manifold, and $H\,=\, \{e\}$, Theorem
\ref{ti} was proved in \cite{Bi}. 
The proof of Theorem~\ref{ti} given here is an application
of the splitting theorem of \cite{AndKah}, \cite{O} and \cite{O19}.
It is simpler and much shorter than the proof given in \cite{Bi}. It resembles \cite{An}
in the sense that it constitutes a simple category theoretic proof of a result proved using
elaborate analytical methods.

\section{Finite bundles and reduction of structure group}

Throughout, all spaces are complex analytic, and all maps are holomorphic.

By a tensor category we mean a $\C$\nd linear category with $\C$\nd bilinear
tensor product, together with a unit $\!$ and associativity and commutativity constraints satisfying
the usual compatibilities \cite[p.~105, Definition~1.1]{DMOS}.
A tensor functor between tensor categories is a $\C$\nd linear functor together 
with constraints ensuring that unit and tensor product are preserved up to isomorphism \cite[pp.~113--114, Definition~1.8]{DMOS}.

Let $X$ be a complex analytic space.
We identify holomorphic vector bundles on $X$ with locally free $\sO_X$\nd modules of finite type.
The tensor category of holomorphic vector bundles over $X$ will be denoted by $\Mod(X)$.
If $X$ is a point, then $\Mod(X)$ is the category $\Mod(\C)$ of finite-dimensional $\C$\nd vector spaces.

Let $G$ be a finite group.
We write $\Rep_{\C}(G)$ for the tensor category of finite dimensional representations of $G$.
If we regard the algebra $\C[G]$ of functions on $G$ as the left regular representation of $G$, 
where any $g\,\in\, G$ sends any $w$ in $\C[G]$ to $w(g^{-1}\cdot)$,
then $\C[G]$ is a commutative $G$\nd algebra under pointwise multiplication.
The action of $G$ on a representation $V$ of $G$ is an isomorphism
\begin{equation}\label{e:Gact}
V_0 \otimes_{\C} \C[G] \,\stackrel{\sim}{\longrightarrow}\, V \otimes_{\C} \C[G]
\end{equation}
of modules over $\C[G]$ in $\Rep_{\C}(G)$, where $V_0$ is the underlying $\C$\nd vector space for $V$ 
equipped with the trivial action of $G$.

Let $H$ be a complex Lie group, and let $X$ be a complex analytic space with a holomorphic action of $H$. 
An $H$\nd equivariant vector bundle over $X$ is a vector bundle over $X$ equipped with
an action of $H$ above that on $X$. If ${\mathcal V}$ is such a vector bundle we write
\begin{equation*}
H^0_H(X,\,{\mathcal V})
\end{equation*}
for the $\C$\nd vector space of $H$\nd invariant cross-sections of ${\mathcal V}$, \emph{i.e.} those cross-sections 
whose pullback onto $H \times X$ along the projection to $X$ coincides with the pullback by the
map $H \times X\, \longrightarrow\, X$ giving the action of $H$ on $X$.

We denote by $\Mod_H(X)$ the tensor category of $H$\nd equivariant vector bundles over $X$, where a morphism
is a morphism of vector bundles compatible with the actions of $H$,
and the tensor product is the tensor product of vector bundles equipped with an appropriate action of $H$.
Given also $H'$ acting on $X'$, together with a homomorphism $$H'\,\longrightarrow\, H$$ and a morphism $X'
\,\longrightarrow\, X$ compatible with the actions,
pullback along $X'\,\longrightarrow\, X$ defines a tensor functor from $\Mod_H(X)$ to $\Mod_{H'}(X')$.
When $H$ is a finite group $\Gamma$ and $X$ is a point, the tensor category $\Mod_H(X)$ is $\Rep_{\C}(\Gamma)$.

Let $G$ be a finite group.
If $P$ is a holomorphic principal $G$\nd bundle over $X$, there is an action of $G$ on $P$ with $gp
\,=\, pg^{-1}$ for $g$ in $G$ and $p$ in $P$.
Thus we have pullback tensor functors from $\Rep_{\C}(G)$ and from $\Mod(X)$ to $\Mod_G(P)$.
Given $V$ in $\Rep_{\C}(G)$, a pair consisting of ${\mathcal V}$ in $\Mod(X)$ and a $G$\nd isomorphism from the pullback
of $V$ onto $P$ to the pullback of ${\mathcal V}$ onto $P$ exists, and is unique up to unique isomorphism:
for the existence reduce by arguing locally over $X$ to the case where $P$ is trivial.
We write
\begin{equation}\label{t1}
P \times^G V
\end{equation}
for ${\mathcal V}$. 
Formation of $P \times^G V$ commutes with pullback of principal $G$\nd bundles.

By an $H$\nd equivariant principal $G$\nd bundle over $X$ we mean a principal $G$\nd bundle $P$ over $X$ together 
with an action of $H$ on $P$ above that on $X$, such that $(hp)g\,=\,h(pg)$ for points $h$, $p$ and $g$ of 
respectively $H$, $P$ and $G$. There is an action of $H \times G$ on such a $P$ above the action of $H$ on $X$, 
with $(h,\,g)p \,=\, hpg^{-1}$. Again we have pullback tensor functors from $\Rep_{\C}(G)$ and from $\Mod_H(X)$ to 
$\Mod_{H \times G}(P)$. Given $V$ in $\Rep_{\C}(G)$, a pair consisting of ${\mathcal V}$ in $\Mod_H(X)$ and an $(H 
\times G)$\nd isomorphism from the pullback of $V$ onto $P$ to the pullback of ${\mathcal V}$ onto $P$ exists, and 
is unique up to unique isomorphism. Indeed ${\mathcal V}$ is $P \times^G V$ in \eqref{t1} with an appropriate 
$H$\nd equivariant structure. We then have a tensor functor $P \times^G -$ from $\Rep_{\C}(G)$ to $\Mod_H(X)$.

A commutative $\sO_X$\nd algebra ${\mathcal R}$ will be called finite \'etale if locally on $X$ it is a 
finite product of copies of $\sO_X$.
It is equivalent to require that the underlying $\sO_X$\nd module of ${\mathcal R}$ lie in $\Mod(X)$ and
that each fibre of ${\mathcal R}$ be a finite product of copies of $\C$.
Indeed since $\sO_{X,x}$ is Henselian for each $x$, a decomposition of the fibre of ${\mathcal R}$ at $x$ 
into a product of copies of $\C$ will lift to a decomposition of the stalk of ${\mathcal R}$ at $x$,
and the idempotents defining it will then extend to some neighbourhood of $x$. 

A morphism $X' \,\longrightarrow\, X$ of complex analytic spaces will be called finite \'etale if
locally on $X$ it is isomorphic to the projection onto $X$ of the product of $X$ with a finite set. 
A surjective finite \'etale morphism will also be called a finite \'etale covering of $X$.
The assignment to $X'$ with structural morphism $a$ of the commutative $\sO_X$\nd algebra $a_*\sO_{X'}$
defines an anti-equivalence from finite \'etale complex analytic spaces over $X$ to finite \'etale
$\sO_X$\nd algebras, and from $H$\nd equivariant finite \'etale complex analytic spaces over $X$ 
to $H$\nd equivariant finite \'etale $\sO_X$\nd algebras.

We say that $X$ is $H$\nd connected if it is not the disjoint union of two non-empty $H$\nd equivariant open 
subspaces. It is equivalent to require that $H^0_H(X,\,\sO_X)$ have no idempotent other than $0$ or $1$. If $X$ is 
$H$\nd connected, any $H$\nd equivariant vector bundle on $X$ has constant rank. 

\begin{lem}\label{l:equiv}
Let $H$ be a complex Lie group and $X$ be a non-empty complex analytic space with an action of $H$.
Then the following conditions are equivalent:
\begin{enumerate}
\item\label{i:const}
the restriction to $X_{\mathrm{red}}$ of any $H$\nd invariant regular function on $X$ is constant;
\item\label{i:local}
$H^0_H(X,\,\sO_X)$ is a local $\C$\nd algebra;
\item\label{i:hens}
$H^0_H(X,\,\sO_X)$ is a Henselian local $\C$\nd algebra with residue field $\C$;
\item\label{i:lim}
$H^0_H(X,\,\sO_X)$ is the limit of a sequence $(R_n)_{n\ge1}$ of commutative $\C$\nd algebras with surjective transition homomorphisms 
and with $R_1 \, = \, \C$, such that for each $n$ the kernel of $R_n \, \longrightarrow \, R_1$ is nilpotent of degree at most~$n$.
\end{enumerate}
If these conditions hold, then they hold with $X$ replaced by any $H$\nd connected $H$\nd equivariant finite \'etale cover.
\end{lem}

\begin{proof}
Write $R \, = \, H^0_H(X,\,\sO_X)$.
If ${\mathcal N}$ is the nilradical of $\sO_X$, then for each $n$ the action of $H$ on $X$ induces one on
the complex analytic subspace $X_n$ of $X$ with the same underlying topological space and structure sheaf
$\sO_X/{\mathcal N}^n$,
with $X_1 \, = \, X_{\mathrm{red}}$.
Since ${\mathcal N}$ is locally nilpotent on $X$,
\begin{equation*}
R \,=\, \lim_n H^0_H(X_n,\,\sO_{X_n}),
\end{equation*}
where the kernel of $H^0_H(X_n,\,\sO_{X_n})\,\longrightarrow\, H^0_H(X_1,\,\sO_{X_1})$ is nilpotent of degree
at most $n$.
Thus \eqref{i:const} implies \eqref{i:lim} with $R_n$ the image of
\begin{equation*}
R \,\longrightarrow\, H^0_H(X_n,\,\sO_{X_n}).
\end{equation*}
If \eqref{i:lim} holds, then any finite commutative $R$\nd algebra $R'$ which is a free $R$\nd module is the limit of 
the $R'{}\!_n \, = \, R' \otimes_R R_n$.
Since the transition homomorphisms of $(R'{}\!_n)$ are surjective with nilpotent kernel, it follows that any idempotent of $R'{}\!_1$
lifts uniquely to an idempotent of each $R'{}\!_n$, and hence to $R'$, so that \eqref{i:hens} holds.
That \eqref{i:hens} implies \eqref{i:local} is clear.
If \eqref{i:local} holds, then the evaluation homomorphism $R \, \longrightarrow \, \C$ at each point of $X$
has kernel the unique maximal ideal of $R$, so that \eqref{i:const} holds.

Let $X'$ be an $H$\nd connected $H$\nd equivariant finite \'etale cover of $X$. 
Then the push forward $a_*\sO_{X'}$ along $a:X' \, \longrightarrow \, X$ is locally free of constant rank over $X$. 
For any element $\alpha$ of $H^0_H(X',\,\sO_{X'})$, the Cayley--Hamilton theorem applied to the endomorphism of $a_*\sO_{X'}$ induced by $\alpha$
thus shows that $\alpha$ is annihilated by a monic polynomial with coefficients in $R$.
Suppose that \eqref{i:const} holds.
Then the restriction $\overline{\alpha}$ of any $\alpha$ to $X'{}\!_\mathrm{red}$ is annihilated by a monic polynomial 
with coefficients in $\C$.
Since $X'{}\!_\mathrm{red}$ is reduced and $H$\nd connected, it follows that $\overline{\alpha}$ lies in $\C$.
Thus \eqref{i:const} holds with $X'$ for $X$. 
\end{proof}

Let $X$ be complex analytic space with an action of $H$.
If $Y$ and $Z$ are $H$\nd 
equivariant finite \'etale coverings of $X$, and $Y$ is $H$\nd connected, then any two $H$\nd equivariant morphisms from $Y$ to $Z$ over $X$
which send a given point $y$ of $Y$ to the same point of $Z$ actually coincide over the whole of $Y$, because their 
equaliser is an open and closed $H$\nd equivariant subset of $Y$ containing $y$.

Suppose that $X$ is $H$\nd connected, and let $X'$ be an $H$\nd equivariant finite \'etale covering of $X$.
If $x_1,\,x_2,\, \ldots ,\, x_n$ are the points of $X'$ above the point $x$ of $X$, write
$\mathcal P$ for the $H$\nd connected component
that contains the point
\begin{equation*}
p \,= \,(x_1,\,x_2,\, \ldots ,\,x_n).
\end{equation*}
of the fibre product over $X$ of $n$ copies of $X'$.
For each point $q$ of $\mathcal P$ above $x$, there exists an $H$\nd equivariant endomorphism of $\mathcal P$ over $X$
sending $p$ to $q$, because for each $i$ there exists an $H$\nd equivariant morphism from $\mathcal P$ to $X'$ over $X$
sending $p$ to $x_i$.
Since $\mathcal P$ is $H$\nd connected, such a endomorphism is unique, and is an automorphism.
The group $G$ of $H$\nd equivariant automorphisms of $\mathcal P$ over $X$ thus acts simply transitively 
on the fibre of $\mathcal P$ above $x$ and hence on every fibre.
Using the inverse involution of the finite group $G$, we obtain an $H$\nd connected $H$\nd equivariant principal
$G$\nd bundle $\mathcal P$ from which there exists an $H$\nd equivariant morphism to $X'$ over $X$.

An $H$\nd equivariant vector bundle over $X$
will be called $H$\nd trivial if it is $H$\nd isomorphic to the pullback onto $X$ of a finite-dimensional
$\C$\nd vector space $V$ on which $H$ acts trivially. 

\begin{lem}\label{l:triv}
Let $H$ be a complex Lie group and $X$ be a non-empty complex analytic space with an action of $H$ such that
the equivalent conditions of Lemma~\ref{l:equiv} hold.
Then the following conditions on an $H$\nd equivariant vector bundle ${\mathcal V}$ over $X$ are equivalent:
\begin{enumerate}
\item\label{i:cov}
There exists an $H$\nd equivariant finite \'etale covering $X'$ of $X$ such that the pullback of
${\mathcal V}$ onto $X'$ is $H$\nd trivial.

\item\label{i:bun}
The $H$\nd equivariant vector bundle ${\mathcal V}$ is $H$\nd isomorphic to $P\times^G V$ for some finite group $G$, finite-dimensional representation $V$ of $G$,
and $H$\nd equivariant principal $G$\nd bundle $P$ over $X$. 
\end{enumerate}
\end{lem}

\begin{proof}
If \eqref{i:bun} holds then \eqref{i:cov} holds with $X' = P$.

Conversely suppose that \eqref{i:cov} holds. Since $X$ is $H$\nd connected, there is 
as above a finite group $G$ and an $H$\nd connected $H$\nd equivariant principal $G$\nd bundle $P$ (denoted above 
by $\mathcal P$) over $X$ with an $H$\nd equivariant morphism to $X'$ over $X$. The pullback $a^*{\mathcal V}$ of 
${\mathcal V}$ along $a\,:\,P \,\longrightarrow\, X$ is thus $H$\nd trivial. To see that \eqref{i:bun} holds, it 
is to be shown that for some finite-dimensional representation $V$ of $G$ there exists an $(H \times G)$\nd equivariant isomorphism from the pullback $V \otimes_{\C} 
\sO_P$ of $V$ onto $P$ to $a^*{\mathcal V}$.

The equivalent conditions of Lemma~\ref{l:equiv} hold with $X$ replaced by $P$.
It follows that
\begin{equation*}
R \,=\, H^0_H(P,\,\sO_P)
\end{equation*}
is the limit of a sequence $(R_n)_{n\ge1}$ of $\C$\nd algebras as in \eqref{i:lim} of Lemma~\ref{l:equiv}.
Let
\begin{equation*}
\rho\,:\,G \,\longrightarrow\, {\rm GL}_n(R)
\end{equation*}
be a group homomorphism, and write $\rho_r$ for $\rho$ followed by the projection onto ${\rm GL}_n(R_r)$.
The group $G$ acts on the kernel $J_{r+1}$ of the projection from ${\rm GL}_n(R_{r+1})$ to ${\rm GL}_n(R_r)$, with $g$ in $G$
acting as conjugation by the element $\rho_1(g)$ of the subgroup ${\rm GL}_n(\C)$ of ${\rm GL}_n(R_{r+1})$.
Since $G$ is finite and $J_{r+1}$ is the additive group of a $\C$\nd vector space, we have
\begin{equation*}
H^1(G,\,J_{r+1})\, =\, 0.
\end{equation*}
Thus there is an element $\theta \,=\, (\theta_r)$ of ${\rm GL}_n(R) \,=\, \lim_r {\rm GL}_n(R_r)$ such that the 
conjugate of $\rho$ by $\theta$ factors through ${\rm GL}_n(\C)$, because the obstruction to lifting a given 
$\theta_r$ to a $\theta_{r+1}$ lies in $H^1(G,\,J_{r+1})$.

The $G$\nd equivariant structure of $a^*{\mathcal V}$ defines an action of $G$ on the finitely generated
free $R$\nd module $H^0_H(P,\,a^*{\mathcal V})$. Choosing a basis
$e_1,\, e_2,\, \cdots ,\,e_n$ of $H^0_H(P,\,a^*{\mathcal V})$ thus defines a homomorphism $\rho$ as above.
If we modify the basis by a $\theta$ in ${\rm GL}_n(R)$ such that the conjugate of $\rho$ by $\theta$ factors through
${\rm GL}_n(\C)$,
then the action of $G$ on $H^0_H(P,\,a^*{\mathcal V})$ restricts to an action on the $\C$\nd vector subspace $V$
generated by the $e_i$.
The embedding of $V$ then induces the required $(H \times G)$\nd equivariant isomorphism.
\end{proof}

Let ${\mathcal C}$ be a tensor category which is rigid (\emph{i.e.} every object has a dual).
Then the trace
\begin{equation*}
{\rm tr}(e) \,\in\, \End_{\mathcal C}(\I)
\end{equation*}
of an endomorphism $e$ in ${\mathcal C}$ is defined.
If the $\C$\nd algebra $\End_{\mathcal C}(\I)$ is local,
then the morphisms $r\,:\,A \,\longrightarrow\, B$ in
${\mathcal C}$ for which ${\rm tr}(s \circ r)$ lies in the maximal ideal of $\End_{\mathcal C}(\I)$
for every $s\,:\,B \,\longrightarrow\, A$ form the unique maximal tensor ideal ${\mathcal J}$ of
${\mathcal C}$ \cite[p.~73]{O19}. We then write
\begin{equation}\label{oc}
\overline{\mathcal C}
\end{equation}
for the quotient ${\mathcal C}/{\mathcal J}$.
It has the same objects as ${\mathcal C}$, and there is a projection ${\mathcal C}\,\longrightarrow\,
\overline{\mathcal C}$ which is the identity on objects and is a full tensor functor with kernel ${\mathcal J}$.

The group algebra of the symmetric group of degree $d$ acts on the $d$-th tensor power of every object of ${\mathcal C}$,
and if ${\mathcal C}$ is pseudo-abelian
we may define for example the $d$-th exterior power as the image of the antisymmetrising idempotent.

Suppose that ${\mathcal C}$ is essentially small (\emph{i.e.} has a small skeleton) and pseudo-abelian, that the 
$\C$\nd algebra $\End_{\mathcal C}(\I)$ is Henselian local with residue field $\C$, and that for every object of 
${\mathcal C}$ some exterior power is $0$. Then $\overline{\mathcal C}$ is a semisimple Tannakian category over 
$\C$ (\cite[Proposition 11.2(i)]{O19} and \cite[p.~165, Th\'eor\`eme 7.1]{De90}), and the projection ${\mathcal 
C}\,\longrightarrow\, \overline{\mathcal C}$ reflects isomorphisms \cite[Proposition 11.1(ii)]{O19}. 
It follows that $\mathcal C$ has the Krull--Schmidt property (\emph{i.e.} the commutative monoid under direct sum of isomorphism classes 
of objects of $\mathcal C$ is freely generated by the classes of the indecomposable objects) because $\overline{\mathcal C}$ does.
The splitting theorem for ${\mathcal C}$ (\cite[p.~225, Th\'eor\`eme 16.1.1(a)]{AndKah}, \cite[Theorem~1.1]{O}, \cite[Theorem~11.7]{O19}) states 
that the projection has a right quasi-inverse, \emph{i.e.} a tensor functor 
$\overline{\mathcal C}\,\longrightarrow\,{\mathcal C}$ with the composite
\[
\overline{\mathcal C}\,\longrightarrow\,{\mathcal C}\,\longrightarrow\, \overline{\mathcal C}
\]
tensor isomorphic to the identity. In \cite{AndKah} and \cite{O} 
this result is proved under the stricter hypothesis that $\End_{\mathcal C}(\I)\,= \,\C$. A reader who wishes to 
confine attention to this case may replace the hypothesis on $H$ and $X$ in Theorem~\ref{t:main} by the stricter 
one that $H^0_H(X,\,\sO_X) \, = \, \C$.
However it is still convenient to prove Lemmas~\ref{l:equiv} and \ref{l:triv}
in the generality given, because unless $X$ is reduced the condition $H^0_H(X,\,\sO_X) \, = \, \C$ is not stable under 
passage to $H$\nd connected $H$\nd equivariant finite \'etale covers.

Let $H$ and $X$ be as in Lemma~\ref{l:triv}.
Then $X$ is $H$\nd connected so that every $\mathcal V$ in $\Mod_H(X)$ has constant rank.
Since $\End(\I)$ in $\Mod_H(X)$ is $H^0_H(X,\,\sO_X)$, it thus follows from \eqref{i:hens} of Lemma~\ref{l:equiv} that the above 
conditions on $\mathcal C$ are satisfied when $\mathcal C$ is $\Mod_H(X)$.

Reducing using \eqref{i:const} of Lemma~\ref{l:triv} to the $H$\nd trivial case shows that 
the full subcategory of $\Mod_H(X)$ consisting of the $\mathcal V$ satisfying the conditions of Lemma~\ref{l:triv} 
is stable under the formation of tensor products, duals, direct sums and (since $\End(\I)$ is local) direct summands.
When $X$ is reduced, the full tensor subcategory of such $\mathcal V$ is Tannakian over $\C$, with
each point of $X$ defining a fibre functor,
because by Lemma~\ref{l:equiv} we then have $\End(\I) \, = \, \C$ in $\Mod_H(X')$ for any $H$\nd connected $H$\nd equivariant 
\'etale cover $X'$ of $X$, so that any morphism between $H$\nd trivial objects in $\Mod_H(X')$ factors as the projection
onto a direct summand followed by the embedding of a direct summand.

If ${\mathcal V}$ is an $H$\nd equivariant vector bundle over $X$ and $f(t)\,=\,\sum_{i=0}^n 
a_it^i$ is a polynomial whose coefficients $a_i$ are nonnegative integers, define the $H$\nd equivariant vector 
bundle $f({\mathcal V})$ as $\bigoplus_{i=0}^n ({\mathcal V}^{\otimes i})^{\oplus a_i}$; note that ${\mathcal 
V}^{\otimes 0}\,=\,{\mathcal O}_X$ and $({\mathcal V}^{\otimes i})^{\oplus 0}\,=\, 0$. Then ${\mathcal V}$ will 
be called $H$\nd finite if there exist distinct polynomials $f_1$ and $f_2$ with $f_1({\mathcal V})$ $H$\nd 
isomorphic to $f_2({\mathcal V})$.
Since $\Mod_H(X)$ has the Krull--Schmidt property, it is equivalent to require that the ${\mathcal V}^{\otimes n}$
should contain only finitely many non-isomorphic indecomposable direct summands in $\Mod_H(X)$.
Thus the full subcategory of $H$\nd finite $\mathcal V$ in $\Mod_H(X)$ is stable under tensor products, duals, direct sums and direct summands.

\begin{thm}\label{t:main}
Let $H$ and $X$ be as in Lemma~\ref{l:triv}.
Then an $H$\nd equivariant vector bundle ${\mathcal V}$ over $X$ is $H$\nd finite if and only if it satisfies
the equivalent conditions of Lemma~\ref{l:triv}.
\end{thm}

\begin{proof}
Suppose that ${\mathcal V}$ satisfies \eqref{i:bun} of Lemma~\ref{l:triv}.
Then class of $V$ in the representation ring of $G$ is annihilated by a 
non-zero polynomial with coefficients in ${\mathbb Z}$, because the representation ring is finite over ${\mathbb Z}$.
Thus $P \times^G V$ and hence ${\mathcal V}$ is $H$\nd finite.

Conversely if ${\mathcal V}$ is $H$\nd finite, we show that ${\mathcal V}$ satisfies \eqref{i:cov} of Lemma~\ref{l:triv}.
Denote by ${\mathcal C}$ the full pseudo-abelian rigid tensor subcategory of $\Mod_H(X)$ generated by ${\mathcal V}$.
Then $\overline{\mathcal C}$ (defined in \eqref{oc}) is semisimple Tannakian over $\C$.
The image $\overline{\mathcal V}$ of $\mathcal V$ in $\overline{\mathcal C}$ generates it as a semisimple abelian tensor category.
Thus for some (not necessarily connected) reductive algebraic group $G$ over $\C$ there is by \cite[p.~130--131, Theorem 2.11]{DMOS}
a tensor equivalence
\begin{equation}\label{e:equ}
\Rep_{\C}(G) \,\longrightarrow\, \overline{\mathcal C}.
\end{equation}
Since $\mathcal C$ contains only finitely many non-isomorphic indecomposable objects by $H$\nd finiteness of $\mathcal V$,
the same holds for $\overline{\mathcal C}$ and $\Rep_{\C}(G)$. 
Thus $G$ is finite.

By the splitting theorem, the projection ${\mathcal C}\,\longrightarrow\,
\overline{\mathcal C}$ has a right quasi-inverse
\begin{equation*}
\overline{\mathcal C} \,\longrightarrow\,{\mathcal C}.
\end{equation*}
The image of $\overline{\mathcal V}$ in ${\mathcal C}$ is isomorphic to ${\mathcal V}$, because the projection is full and reflects isomorphisms.
Composing with \eqref{e:equ} and the embedding into $\Mod_H(X)$
gives a tensor functor 
\begin{equation*}
T \,:\,\Rep_{\C}(G)\,\longrightarrow\, \Mod_H(X)
\end{equation*}
such that $\mathcal V$ is
$H$\nd isomorphic to $T(V)$ for some $V$ in $\Rep_{\C}(G)$.

Applying $T$ to \eqref{e:Gact} gives an $H$\nd equivariant isomorphism of $T(\C[G])$\nd modules
\begin{equation}\label{e:TGact}
T(V_0) \otimes_{\sO_X} T(\C[G]) \,\stackrel{\sim}{\longrightarrow}\, T(V) \otimes_{\sO_X} T(\C[G]),
\end{equation}
where $T(V_0)$ is $H$\nd trivial. Since $T$ followed by passage to the fibre at a point of $X$ is tensor 
isomorphic to the forgetful functor, the fibres of $T(\C[G])$ are isomorphic to the underlying $\C$\nd algebra of 
$\C[G]$. Hence $T(\C[G])$ is a non-zero $H$\nd equivariant finite \'etale $\sO_X$\nd algebra.
There thus exists an 
$H$\nd equivariant finite \'etale covering $a:X' \,\longrightarrow\, X$ of $X$ with $a_*\sO_{X'}$ $H$\nd 
isomorphic to $T(\C[G])$. Replacing $T(\C[G])$ by $a_*\sO_{X'}$ in $\eqref{e:TGact}$ and applying $a^*$ and then 
$a^*a_*\sO_{X'} \,\longrightarrow\, \sO_{X'}$ shows that $a^*T(V)$ and hence $a^*{\mathcal V}$ is $H$\nd trivial.
\end{proof}

It follows from Theorem~\ref{t:main} and the remarks preceding it that if $H$ and $X$ are as in Lemma~\ref{l:triv} and $X$ is reduced, 
then the full tensor subcategory of $\Mod_H(X)$ consisting of the $H$\nd finite objects is a neutral Tannakian category over $\C$, and hence 
by the $H$\nd finiteness is tensor equivalent to $\Rep_{\C}(G)$ for a profinite group $G$.
It can be shown more generally using the structure theorem of \cite[\S~7]{O} that a tensor category $\mathcal C$ as above 
is tensor equivalent to $\Rep_{\C}(G)$ for a profinite group $G$ if and only if $\End_{\mathcal C}(\I) \, = \, \C$,
every object of $\mathcal C$ satisfies a condition analogous to $H$\nd finiteness, and $\mathcal C$ is reduced, in the sense that
$h^{\otimes n} \, = \, 0$ implies $h \, = \, 0$ for every morphism $h$ of $\mathcal C$.

Much more general results than Theorem~\ref{t:main} can be obtained by combining the splitting theorem of 
\cite{AndKah}, \cite{O} and \cite{O19} with the complex analytic analogue of the well-known dictionary between 
principal bundles and tensor functors from categories of group representations to categories of vector bundles. In 
particular we obtain in this way complex analytic analogues of the results on uniqueness up to conjugacy of 
minimal reductions of principal bundles with reductive structure group proved by Bogomolov \cite[p.~401, Theorem 
2.1]{Bo} and in a more general form in \cite[Corollary 12.13]{O19}.

Similar techniques are potentially applicable in many other geometric contexts. As well as rigid analytic analogues, 
we may for example consider vector bundles or principal bundles equipped with a connection along the leaves of a 
foliated differentiable manifold on which every leafwise constant function is constant. For the analogue of 
Theorem~\ref{t:main} in the latter case, the splitting theorem again gives an algebra $R$ in an $\mathbb R$\nd 
linear tensor category of vector bundles with connection along the leaves of the foliated manifold. The fibres of 
$R$ are isomorphic to ${\mathbb R}[G]$ for a finite, but not necessarily discrete, algebraic group $G$ over $\mathbb 
R$, and the points of the required \'etale cover are the $\mathbb R$\nd points of their spectra.

%%%%%%%%%%%%%%%%%%%%%
% References
%%%%%%%%%%%%%%%%%%%%%

\end{document}